\documentclass{amsart}
\usepackage{graphicx,amsfonts,amssymb,amsmath,amsthm}
\usepackage[pdftex]{hyperref}  
\hypersetup{pdfstartview={XYZ 120 670 1}, pdfpagemode=none}
\usepackage[hyperpageref]{backref} 
\usepackage{cite}  
\usepackage{xcolor}

\theoremstyle{plain} 
\newtheorem{theorem}    {Theorem}[section] 
\newtheorem{lemma}      [theorem]{Lemma}

\theoremstyle{definition}

\theoremstyle{remark}
\newtheorem{remark}              {Remark}

\newtheorem*{notation}            {Notation}

\numberwithin{equation}{section}

\def\C{\mathbb C}

\def\R{\mathbb R}

\makeatletter
\@namedef{subjclassname@2020}{
  \textup{2020} Mathematics Subject Classification}
\makeatother

\begin{document}

\title[On the occurrence of Hecke eigenvalues in sectors]{On the occurrence of Hecke eigenvalues \\ in sectors} 

\author{Nahid Walji}

\maketitle
\begin{abstract} Let $\pi$ be a non-self-dual unitary cuspidal automorphic representation not of solvable polyhedral type for ${\rm GL}(2)$ over a number field. We show that $\pi$ has a positive upper Dirichlet density of Hecke eigenvalues in any sector whose angle is at least 2.63 radians. 
\end{abstract}

\section{Introduction}
Let $\pi$ be a unitary cuspidal automorphic representation for ${\rm GL}(2)$ over a number field $F$. We assume that it is not of solvable polyhedral type, which means that it does not correspond to an Artin representation of dihedral, tetrahedral, or octahedral type. Associated to a finite place $v$ where $\pi$ is unramified, we have the multiset of Satake parameters $\{\alpha_v (\pi), \beta_v (\pi)\}$ and their sum is called the Hecke eigenvalue $a_v(\pi)$ of $\pi$ at $v$.

One can ask about the distribution of the sequence $(a_v(\pi))_v$.
If one restricts to automorphic representations that correspond to holomorphic forms, then more is known. For example the Sato-Tate conjecture has been proved for a wide range of Hilbert modular forms~\cite{ST}. In the general case however, much less is known.
For example, in an appendix to~\cite{Sh94}, J.-P.~Serre asked if, for self-dual $\pi$, it can be shown that there are infinitely many Hecke eigenvalues greater than a given positive constant $c$ (and similarly, if there are infinitely many Hecke eigenvalues less than a given negative constant $c'$). An answer to this was provided by Theorem~1.2 of~\cite{Wa18} with $c = 0.905$ and $c' = -1.164$.

In the case of when $\pi$ is a non-self-dual, one can extend the question as follows: For what angle $\theta$ can it be shown that there are infinitely many Hecke eigenvalues in any sector of size $\theta$? Furthermore, given any such sector, for what $c$ do we have infinitely many Hecke eigenvalues greater than size $c$?

A consequence of Theorem~1.3 of~\cite{Wa18} is that this holds true for $\theta = \pi$ radians, with $c = 0.5$. In this paper, we will improve the value of $\theta$ to 2.63 radians and improve $c$ to $0.595$.
\begin{theorem}\label{t1}
Let $\pi$ be a non-self-dual unitary cuspidal automorphic representation for ${\rm GL}(2)/F$, where $F$ is a number field, that is not of solvable polyhedral type. Then, for any angle $\phi$ we have that the following set of places 
\begin{align*}
\{v \mid {\rm arg}(a_v(\pi)) \in (\phi -1.314, \phi + 1.314)\}
\end{align*}
has positive upper Dirichlet density.
Furthermore, the subset of such places whose associated Hecke eigenvalue has a size of at least 0.595 also has positive upper Dirichlet density.
\end{theorem}

\section{Asymptotic properties of certain Dirichlet series}\label{dsec}

In this section, we assume that $\pi$ is a cuspidal automorphic representation for ${\rm GL}(2)/F$ that is not self-dual and not of solvable polyhedral type.

\begin{notation}
Denote by $X = X (\pi)$ the set of archimedean places as well as places at which $\pi$ is ramified.
Values of $k$ will be associated to our examination of the asymptotic behaviour of $$\sum_{v \not \in X} {\rm Re}(e ^{i \phi}a_v(\pi))^k {\rm N}v^{-s},$$
as $s \rightarrow 1^+$,
for $k = 3,4,6,8$, where $\phi$ is any fixed angle in $[0,2 \pi)$. Let $\omega$ be the central character of $\pi$ and denote the order of this character by $r$. Lastly, we will write $\ell (s) := \log (1/(s-1))$.
\end{notation}

We will repeatedly make use of the bounds towards the Ramanujan conjecture of Kim--Sarnak~\cite{Ki03} (in the rational case) and Blomer--Brumley~\cite{BB11} (for number fields). We will also need the functoriality results of Gelbart--Jacquet~\cite{GJ78}, Kim--Shahidi~\cite{KS00,KS02}, and Kim~\cite{Ki03}, regarding the symmetric square, cube, and fourth power lifts of cuspidal automorphic representations for {\rm GL}(2). 

\subsection{$k = 3$}

We consider incomplete $L$-functions of the form $L^X(s, \pi ^m\times \overline{\pi}^n)$ where $m,n$ are non-negative integers and $m+n = 3$.

In the case $(m,n)= (2,1)$, making use of Clebsch--Gordan decompositions and the unitary of $\pi$, we obtain
\begin{align*}
L^X(s, \pi \times \pi \times \overline{\pi}) = L^X(s, {\rm Sym}^3 \pi \otimes \omega ^{-1}) L^X(s, \pi)^2 
\end{align*}
where $\omega$ is the central character of $\pi$. Taking logarithms and using the bounds towards the Ramanujan conjecture~\cite{Ki03,BB11} we obtain 
\begin{align*}
\sum_{v \not \in X}\frac{a_v(\pi) ^2  \overline{a_v(\pi)}}{{\rm N}v^s} = O(1),
\end{align*}
as $s \rightarrow 1^+$.

Using a similar approach for the other cases, we see that the same asymptotic behaviour occurs for $\sum_{v \not \in X}a_v(\pi) ^m \overline{a_v(\pi)^n}  {\rm N}v^{-s} $ for $(m,n) = (3,0), (1,2),$ and $(0,3)$.
Therefore, for any $\phi \in [0,2 \pi)$,
\begin{align*}
\sum_{v \not \in X}\frac{{\rm Re}(e ^{i \phi}a_v(\pi))^3}{{\rm N}v^s} = \frac{1}{2 ^3} &\left(\sum_{v \not \in X}\frac{e ^{3i \phi}a_v(\pi)^3}{{\rm N}v^s}
+3 \sum_{v \not \in X}\frac{e ^{i\phi}a_v(\pi)^2 \overline{a_v(\pi)}}{{\rm N}v^s} \right. \\
&+ \left. 3\sum_{v \not \in X}\frac{e ^{-i\phi}a_v(\pi) \overline{a_v(\pi)^2}}{{\rm N}v^s}
+ \sum_{v \not \in X}\frac{e ^{-3i \phi} \overline{a_v(\pi)^3}}{{\rm N}v^s}\right)\\
= O(1) &
\end{align*}
since each of the four series on the right-hand side is bounded as $s \rightarrow 1^+$.

\subsection{$k = 4$:} \label{sseck4} Using the same approach as in the $k = 3$ case, we find that 
\begin{align*}
L^X(s, \pi \times \pi \times \pi \times \pi) = L^X(s, {\rm Sym}^4 \pi)L^X(s, {\rm Sym}^2 \pi \otimes \omega)^3 L^X(s, \omega ^2)^2.
\end{align*}
Therefore, if $\pi$ has central character of order two,
$$\sum_{v \not \in X}a_v(\pi) ^4 {\rm N}v^{-s} = 2 \ell(s) + O(1)$$ as $s \rightarrow 1^+$. If not, then the series is bounded in that limit.

We also note  
\begin{align*}
L^X(s, \pi \times \pi \times \pi \times \overline{\pi}) = L^X(s, {\rm Sym}^4 \pi \otimes \omega ^{-1})L^X(s, {\rm Sym}^2 \pi)^3 L^X(s, \omega)^2,
\end{align*}
which then implies 
\begin{align*}
\sum_{v \not \in X}{a_v(\pi) ^3 \overline{a_v(\pi) }} {{\rm N}v^{-s}} = O(1)
\end{align*}
as $s \rightarrow 1^+$, since $\pi$ is not self-dual.
We similarly obtain
\begin{align*}
\sum_{v \not \in X} {a_v(\pi) ^2 \overline{a_v(\pi)^2 }}{ {\rm N}v^{-s}} = 2 \ell(s) + O(1),
\end{align*}
and we conclude
\begin{align*}
\sum_{v \not \in X}{{\rm Re}(e ^{i \phi}a_v(\pi))^4}{{\rm N}v^{-s}} =  q_4 \cdot \ell(s) + O(1) ,
\end{align*}
where 
\begin{align*}
q_4=q_4 (\pi,\phi) = 
\begin{cases}
 \frac{3 + \cos 4\phi}{4}, &\text{ if }r=2,\\
 \frac 34, &\text{ if } r \geq 3.
\end{cases}
\end{align*}
\subsection{$k = 6$:}\label{sseck6}

Note that the incomplete $L$-function 
$L^X(s, \pi ^{\times m}\times \overline{\pi}^{\times n})$, for non-negative integers $m + n = 6$, can be expressed as
\begin{align}\label{eq1}
 L^X(s, {\rm Sym}^3 \pi \times {\rm Sym}^3 \pi \otimes \omega ^{-n})L^X(s, {\rm Sym}^3 \pi \times \pi \otimes \omega ^{1-n})^4 L^X(s, \pi \times \pi \otimes \omega ^{2-n})^4 
\end{align}
and also as 
\begin{align}\label{eq2}
&L^X(s, {\rm Sym}^4 \pi \times {\rm Sym}^2 \pi \otimes \omega ^{-n})L^X(s, {\rm Sym}^4 \pi \otimes \omega ^{1-n}) \\ \notag 
& \cdot L^X(s, ({\rm Sym}^2 \pi \otimes \omega ^{1-n})\times {\rm Sym}^2 \pi)^3  L^X(s, {\rm Sym}^2 \pi \otimes \omega ^{2-n})^5 L^X(s, \omega ^{3-n})^2.
\end{align}
The first and third $L$-functions in equation~\ref{eq1} are either invertible at $s=1$ or have a simple pole there. The second $L$-function is invertible at $s=1$. Therefore, equation~\ref{eq1} either is invertible at $s=1$, or has a pole of order 1, 4, or 5 there.
The third and fifth $L$-functions in~\ref{eq2} are either invertible at $s=1$ or have a simple pole there. The rest are invertible at $s=1$. Therefore, at $s=1$ equation~\ref{eq2} is either invertible there or has a pole of order 2, 3, or 5.
So $L^X(s, \pi ^{\times m}\times \overline{\pi}^{\times n})$ is either invertible at $s=1$ or has a pole of order 5. In the latter case, this holds if and only if $\omega ^{3-n}= 1$. 

If $r = 2$, then the incomplete $L$-function $L^X(s, \pi ^{\times m}\times \overline{\pi}^{\times n})$ has a pole of order 5 exactly when $n = 1,3,5$. If $r = 3$, then this $L$-function has a pole of order 5 exactly when $n = 0,3,6$. If $r \geq 4$, then it has a pole of order 5 only when $n = 3$.

Denote by $\alpha_v (\pi)$ and $\beta_v (\pi)$ the Satake parameters of $\pi$ at $v$. Taking logarithms and applying the known bounds on the size of the Satake parameters, we obtain:
\begin{align*}
\sum_{v \not \in X}\sum_{t = 1,2} \frac{(\alpha_v(\pi) ^t + \beta_v(\pi) ^t)^6 \omega_v^{-tn}}{{\rm N}v^{st}}=
\begin{cases} O(1), & \text{ if } r = 2 \text{ and } n = 1,3,5,  \\
   & \text{ if } r = 3 \text{ and } n = 0,3,6, \\
& \text{ or if } r \geq 4 \text{ and } n \neq 3. \\
5 \ell(s) + O(1), &\text{ if } r = 2 \text{ and } n = 0,2,4,6, \\
& \text{ if } r = 3 \text{ and } n = 1,2,4,5, \\
& \text{ or if } r \geq 4 \text{ and } n=3. \end{cases}
\end{align*}
So for any angle $\phi \in [0,2 \pi)$, 
\begin{align}\label{bigeq}
\frac{1}{2 ^6}{\sum_{n = 0}^{6}} \ ^6C_n \sum_{v \not \in X}\sum_{t = 1,2} \frac{(\alpha_v(\pi) ^t + \beta_v(\pi) ^t)^6 \omega_v^{-tn} }{{\rm N}v^{st}}e^{i (6-2n)\phi} \\ \notag 
=\begin{cases}
\frac{5}{16}(3 \cos 4 \phi+ 5) \cdot \ell(s)  + O(1), &\text{ if }r = 2, \\
\frac{5}{32}(\cos 6 \phi+ 10) \cdot \ell(s)  + O(1), &\text{ if }r = 3, \\
\frac{25}{16} \cdot \ell(s)  + O(1), &\text{ if }r \geq 4, \\
\end{cases}
 \end{align}
 as $s \rightarrow 1^+$.
 
We also note that the left-hand side of equation \ref{bigeq} above is equal to
\begin{align*}
\frac{1}{2 ^6} \sum_{v \not \in X} \sum_{t = 1 }^{ 2} \frac{(\alpha_v (\pi) ^t + \beta_v (\pi) ^t)^6}{{\rm N}v^{st}}(e^{i \phi} + \omega_v^{-t} e^{-i\phi})^6 .
\end{align*}
Since $$(\alpha_v (\pi) ^t + \beta_v (\pi) ^t)(e^{i \phi} + \omega_v^{-t} e^{-i \phi}) = (\alpha_v (\pi) ^t + \beta_v (\pi) ^t)e^{i \phi}+ (\overline{\alpha_v (\pi) ^t + \beta_v (\pi) ^t} )e^{-i \phi}$$
 we know that 
 \begin{align*}
 \sum_{v \not \in X} \frac{(\alpha_v (\pi) ^2 + \beta_v (\pi) ^2)^6}{{\rm N}v^{2s}}(e^{i \phi} + \omega_v^{-2} e ^{-i\phi})^6 
 \end{align*}
 is non-negative.
We conclude 
\begin{align}\label{k6eq}
\sum_{v \not \in X} \frac{{\rm Re}(a_v(\pi))^6}{{\rm N}v^s} \leq q_6 \cdot \ell(s) + O(1).
\end{align}
where we can choose
\begin{align*}
q_6 = q_6(\pi) = 
\begin{cases}
\frac 5 2, & \text{ if } r = 2, \\
\frac{55}{32}, & \text{ if }r = 3, \\
\frac{25}{16}, & \text{ if }r \geq 4.
\end{cases}
\end{align*}

\subsection{$k=8$:}\label{sseck8}
For non-negative integers $m + n = 8$, we have   
\begin{align}\label{k8mneq}
&L^X(s, \pi ^{\times m} \times \overline{\pi}^{\times n}) \\ \notag 
=& L^X(s, \pi ^{\times 8} \otimes \omega ^{-n})  \\ \notag 
=&L^X(s, {\rm Sym}^4 \pi \times {\rm Sym}^4 \pi \otimes \omega ^{-n})L^X(s, {\rm Sym}^2 \pi \times {\rm Sym}^2 \pi \otimes \omega ^{2-n})^9
L^X(s, \omega ^{4-n})^4 \\ \notag 
&\cdot L^X(s, {\rm Sym}^4 \pi \times {\rm Sym}^2 \pi \otimes \omega ^{1-n})^6
L^X(s, {\rm Sym}^4 \pi \otimes \omega ^{2-n})^4 L^X(s, {\rm Sym}^2 \pi \otimes \omega ^{3-n})^{12}.
\end{align}

$L^X(s, {\rm Sym}^4 \pi \times {\rm Sym}^4 \pi \otimes \omega ^{-n})$ has a simple pole at $s=1$ when $n = 4$. If it has a pole for other values of $n$, then either it means that ${\rm Sym}^4 \pi$ admits a self-twist, or $\omega$ has order less than or equal to four.
Since there is no known characterisation of when ${\rm Sym}^4 \pi$ admits a self-twist, we examine different cases in terms of the possible order of the central character. If we assume that $L^X(s, {\rm Sym}^4 \pi \times {\rm Sym}^4 \pi \otimes \omega ^{-n})$ has a simple pole at $s=1$, then
$${\rm Sym}^4 \pi \otimes \omega ^{-n} \simeq \widetilde{{\rm Sym}^4 \pi}.$$ Considering the central characters of each side, we obtain $\omega ^{10-5n} = \omega ^{-10}$ and so $\omega$ has order dividing $(20-5n)$.

At this stage, we consider all the different possible pairs of values of  $(r,n)$ for which the incomplete $L$-function $L^X(s, {\rm Sym}^4 \pi \times {\rm Sym}^4 \pi \otimes \omega ^{-n})$ may have a (simple) pole. We mention a few cases explicitly here:
If $r = 2$, then $L^X(s, {\rm Sym}^4 \pi \times {\rm Sym}^4 \pi \otimes \omega ^{-n})$ has a simple pole when $n$ is even and is invertible otherwise. If $r = 3$, then the $L$-function has a pole exactly when $n = 1,4,7$. If $r=4$, there is a pole exactly when $n = 0,4,8$, and if $r = 5$, we cannot rule out the existence of a pole for any value of $n$.

For $L^X(s, {\rm Sym}^2 \pi \times {\rm Sym}^2 \pi \otimes \omega ^{2-n})$, we note that Theorem 2.2.2 of~\cite{KS02} states that for non-dihedral $\pi$, the adjoint lift of $\pi$ admits a self-twist if and only if ${\rm Sym}^3 \pi$ is not cuspidal. However, we have assumed that $\pi$ is not of solvable polyhedral type which means that its symmetric cube lift must be cuspidal, so its adjoint lift, and thus its symmetric square lift, cannot admit a non-trivial self-twist. We now consider the cases of the different values of $r$: If $r = 2$, then the $L$-function has a pole when $n$ is even and is invertible otherwise. If $r = 3$, the $L$-function has a pole exactly when $n = 1,4,7$. If $r = 4$, the $L$-function has a pole exactly when $n = 0,4,8$. Lastly, if $r \geq 5$, then the $L$-function only has a  pole when $n = 4$.

For $L^X(s, \omega ^{4-n})$, the analysis has the exact same outcomes as in the paragraph directly above.

Finally, we note that the last three $L$-functions in equation (\ref{k8mneq}), namely,
\begin{align*}
&L^X(s, {\rm Sym}^4 \pi \times {\rm Sym}^2 \pi \otimes \omega ^{1-n}),\\
&L^X(s, {\rm Sym}^4 \pi \otimes \omega ^{2-n}), \text{ and }\\
&L^X(s, {\rm Sym}^2 \pi \otimes \omega ^{3-n}),
\end{align*}
are all invertible at $s=1$.

We consider 
\begin{align}\label{ank}
A(n,r):= \sum_{v \not \in X}\sum_{t = 1 }^{ 8} \frac{(\alpha_v(\pi) ^t + \beta_v(\pi) ^t)^8 \omega_v^{-tn}}{{\rm N}v^{st}}
\end{align}
If $n = 0$, then from the discussion above on the possible existence (and order) of poles at $s=1$ of the various $L$-functions, we find that equation \ref{ank} is bounded as $s \rightarrow 1^+$ when $r \neq 2,4,5,10,20$. In the case where $r = 2$ or $4$, we have $$A(n,r) =  14 \cdot \ell(s) + O(1),$$ and
in the case where $r = 5,10,$ or $20$, we have $$A(n,r) \leq  \ell(s) + O(1).$$ We proceed similarly in considering other values of $n$ and $r$, recording the asymptotic behaviour of $A(n,r)$ in the table below:
\begin{center}
\begin{tabular}{ |c|c|c| } 
 \hline \textbf{n} & \textbf{r} & \textbf{A(n,r)} \\ 
 \hline \hline 0 or 8 & $ 2,4$& $14 \ell(s) + O(1)$ \\
 \hline  & $ 5,10,20$& $\leq \ell(s) + O(1)$ \\
 \hline  & otherwise 
 & O(1)  \\ 
\hline 1 or 7 & $3$ & $14 \ell(s) + O(1)$ \\
 \hline  & $5,15$ & $\leq \ell(s) + O(1)$ \\
\hline  & otherwise & O(1)  \\
\hline 2 or 6 & $2$ & $14 \ell(s) + O(1)$  \\
 \hline  & $5,10$ & $\leq \ell(s) + O(1)$ \\
 \hline  & otherwise & O(1)  \\
 \hline 3 or 5 & $5$ & $\leq \ell(s) + O(1)$ \\
 \hline  & otherwise & O(1)  \\
  \hline 4 & all & $ 14 \ell(s) + O(1)$  \\
\hline 
\end{tabular}
\end{center}

We can use the above to establish asymptotic bounds on 
\begin{align*}
{\sum_{n = 0}^{8}}  \sum_{v \not \in X}\sum_{t = 1 }^{ 8} \ ^8C_n \frac{(\alpha_v(\pi) ^t + \beta_v(\pi) ^t)^8 \omega_v^{-tn} }{{\rm N}v^{st}}e^{i (8-2n)\phi}.
\end{align*}
We scale the left-hand side of equation above by $1/2^8$ and use positivity to obtain
\begin{align}
 \sum_{v \not \in X} \frac{({\rm Re}(a_v(\pi) e^{i\phi}))^8}{{\rm N}v^s}\label{k8eq}
&\leq  \sum_{v \not \in X} \sum_{t = 1 }^{ 8} 
\frac{({\rm Re}(\alpha_v(\pi)^t e^{i\phi} + \beta_v(\pi)^t e^{i\phi}))^8}{{\rm N}v^{st}} \\
&\leq q_8 \cdot \ell(s) + O(1), \notag 
\end{align}
as $s \rightarrow 1^+$, where 
\begin{align*}
2^8 \cdot q_8=
\begin{cases}
1792 , & \text{ if } r = 2, \\
1204  , & \text{ if } r = 3, \\
1008 , & \text{ if } r = 4 ,\\
1166 , & \text{ if } r = 5, \\
1038 , & \text{ if } r= 10, \\
996 , & \text{ if } r= 15 ,\\
982 , & \text{ if } r = 20,\\
980 , &\text{ otherwise. }
\end{cases}
\end{align*}

\begin{remark}
These bounds appear to be best possible given current knowledge; in particular, there is no known characterisation for when a symmetric fourth power lift from GL(2) admits a self-twist (in contrast to, say, the symmetric square and cube cases, which are well-understood). For context, if we assumed the Ramanujan conjecture, then for any $r \geq 6$ with $r \neq 10,15,20$, the left-hand side of equation (\ref{k8eq}) would have a lower bound of $(980/2^8)\cdot  \ell(s) + O(1)$.
\end{remark}

\section{Bounding subsets of Hecke eigenvalues}
First we recall that the upper and lower Dirichlet densities of a set $S$ of places of a number field $F$ are defined as 
\begin{align*}
\overline{\delta}(S) = \limsup_{s \rightarrow 1^+} \frac{\sum_{v \in S}{\rm N}v^{-s}}{\log (1/(s-1))}
\end{align*}
and 
\begin{align*}
\underline{\delta}(S) = \liminf_{s \rightarrow 1^+} \frac{\sum_{v \in S}{\rm N}v^{-s}}{\log (1/(s-1))},
\end{align*}
respectively, and note that these are equal if and only if the set has a Dirichlet density $\delta (S)$. 

\subsection{Absolute value of Hecke eigenvalues}
The following lemma simply arises from adjusting the proof of Theorem~4.1 from~\cite{KS02}. 
\begin{lemma} \label{kslem} Let $Q =r/s \geq 2$  be a rational number, where $r,s$ are positive integers. Then, for any unitary cuspidal automorphic representation $\pi$ for ${\rm GL}(2)$ over a number field,
we have
\begin{align*}
\overline{\delta}\{v \mid |a_v (\pi)| > Q\} \leq \frac{1}{1 + (Q ^2 -1)^2 + (Q ^4 -3 Q ^2 + 1)^2}.
\end{align*}
\end{lemma}

\begin{proof}
If $\pi$ is of solvable polyhedral type, then we know that it corresponds to an Artin representation~\cite{La80, Tu81} and therefore satisfies the Ramanujan conjecture, so the inequality holds.
If $\pi$ is not of solvable polyhedral type, we know that its adjoint and symmetric fourth power lifts are cuspidal. We construct the following isobaric automorphic representation 
\begin{align*}
\eta = s^4 \alpha \cdot \textbf{1}\boxplus s^4\beta \cdot {\rm Ad}\pi \boxplus s^4 \gamma \cdot (\omega ^{-2}\otimes {\rm Sym}^4 \pi)
\end{align*}
where $\alpha,\beta,\gamma$ are non-negative integers whose values will be determined later.

Now 
\begin{align*}
a_v(\eta)  &= s^4\alpha + s^4\beta (|a_v (\pi)|^2 - 1) + s^4\gamma (|a_v (\pi)|^4  - 3 |a_v(\pi)|^2 +1).
\end{align*}

If $|a_v(\pi)| > Q \geq 2$, then $a_v(\eta) > s^4(\alpha + \beta (Q ^2 -1) + \gamma (Q ^4 -3 Q ^2 +1))$.

For some automorphic representation $\mu$ and non-negative real number $t$, define $T(\mu,t)$ to be the set of finite places $v$ at which $|a_v(\mu)| > t$. From~\cite{Ra97} we know that 
\begin{align*}
\overline{\delta}(T(\mu,t)) \leq \frac{-{\rm ord}_{s=1}L(s,\mu \times \widetilde{\mu})}{t ^2}
\end{align*}

Therefore, since $v \in T(\pi,Q) \Rightarrow v \in T(\eta, s^4\alpha + s^4\beta (Q ^2 -1) + s^4\gamma (Q ^4 -3 Q ^2 +1))$, we have 
\begin{align*}
\overline{\delta}(T(\pi,Q)) \leq \frac{s^8(\alpha ^2 + \beta ^2 + \gamma^2)}{(s^4\alpha + s^4\beta (Q ^2 -1) + s^4\gamma (Q ^4 -3 Q ^2 +1)) ^2}\\
\end{align*}

Choose $\alpha = 1$, $\beta = Q ^2 -1$, and $\gamma = Q ^4 -3 Q^2 + 1$ to get 
\begin{align*}
\overline{\delta}(T(\pi,Q)) \leq \frac{1}{1 + (Q ^2 -1)^2 + (Q ^4 -3 Q ^2 +1)^2}.
\end{align*}

 \end{proof}

\subsection{Case of central characters of order at least 6}
\label{mpf}
From here on, we assume that $\pi$ is non-self-dual and not of solvable polyhedral type, and we will fix an angle $\phi \in [0,2 \pi)$.
We will also make use of the notations $q_4,q_6,$ and $q_8$ from subsections \ref{sseck4}, \ref{sseck6}, and \ref{sseck8}, respectively. Later in the proof we will make the distinction between the cases $r < 6$ and $r \geq 6$.\\

Let 
\begin{align*}
A = A (\pi)  &:= \{v \not \in X\mid {\rm Re}(a_v(\pi)e^{-i \phi})>0 \},\\
B = B (\pi)&:= \{v \not \in X \mid {\rm Re}(a_v(\pi)e^{-i \phi}) \leq 0 \}.
\end{align*}
Given a set $S$ of finite places and a non-negative integer $t$,
we establish the notation
$${\rm ls} (S,t) := 
\limsup_{s \rightarrow 1^+} \left(\frac{\sum_{v \in S}{({\rm Re}(a_v (\pi)e^{-i \phi}))^t}{{\rm N}v^{-s}}}{\log (1/ (s-1))}\right)$$
and similarly 
\begin{align*}
{\rm li} (S,t) := 
\liminf_{s \rightarrow 1^+} \left(\frac{\sum_{v \in S}{({\rm Re}(a_v (\pi)e^{-i \phi}))^t}{{\rm N}v^{-s}}}{\log (1/ (s-1))}\right).
\end{align*}

We also note the following identities that will be referred to later:\\
Given real-valued functions $f,g$ and a point $w \in \R$, we have
\begin{align}\label{lslem}
\limsup_{s \rightarrow w} (f(w) + g(w)) \geq \limsup_{s \rightarrow w} f(w) + \liminf_{s \rightarrow w} g(w) \geq \liminf_{s \rightarrow w} (f(w) + g(w)).
\end{align}
Furthermore, if $f$ and $g$ are non-negative functions, then 
\begin{align}\label{lslem2}
\limsup_{s \rightarrow w} (f(w) \cdot g(w)) \leq \limsup_{s \rightarrow w} f(w) \cdot  \limsup_{s \rightarrow w} g(w).
\end{align}

From subsection~\ref{sseck4} we have that ${\rm ls}(\Sigma_F-X,4)={\rm li}(\Sigma_F-X,4) = q_4$, where $\Sigma_F$ is the set of places of $F$. Applying equation~(\ref{lslem}), we have
\begin{align*}
{\rm li}(A,4)= q_4-{\rm ls} (B,4) .
\end{align*}
We set $d:={\rm ls} (B,4)$. Define
$$S = S (\beta) := \{v \in A \mid ({\rm Re}(a_v (\pi)e ^{-i \phi}))^4  > (q_4-d)\beta\},$$ for some constant $\beta \leq 1$, where we make the assumption that $\overline{\delta}(S) < 1/m$, for some constant $m$.
Note that
\begin{align*}
{\rm li}(A-S,4) \leq (q_4-d)\beta \cdot \underline{\delta}(A-S).
\end{align*}
Using equation~(\ref{lslem}),
\begin{align*}
{\rm li}(A-S,4) + {\rm ls} (S,4) &\geq {\rm li}(A, 4) = q_4 -d \\
{\rm ls} (S, 4) &\geq \left(q_4-d\right)(1- \beta \underline{\delta}(A-S)).
\end{align*}
Applying equations (\ref{k8eq}) and (\ref{lslem2}),
\begin{align*}
{\rm ls} (S, 4)^2 &\leq {\rm ls} (S, 8) \cdot {\rm ls} (S,0) \\
\left(q_4-d\right)^2(1- \beta \underline{\delta}(A-S))^2 & \leq q_8 \cdot \overline{\delta}(S) 
\end{align*}
and from~(\ref{lslem}) we have $$\underline{\delta}(A-S) \leq \overline{\delta}(A) - \overline{\delta}(S),$$ so 
\begin{align}\notag 
\left(q_4-d\right)^2(1- \beta (\overline{\delta}(A) - \overline{\delta}(S)))^2 & \leq q_8 \cdot \overline{\delta}(S) \\
\left(q_4-d\right)^2(1- \beta (1 - \overline{\delta}(S)))^2 & \leq q_8 \cdot \overline{\delta}(S). \label{deseq1}
\end{align}

Now define 
\begin{align*}
T = T(\alpha) := \left\{v \in A \Biggm| ({\rm Re}(a_v (\pi)e ^{-i \phi})) ^3 \geq \alpha d ^{5/4} \left(q_8 - \left(q_4-d\right)^2\right)^{-1/4}\right\}
\end{align*}
for some constant $\alpha \leq 1$, and we make the assumption that  $\overline{\delta}(T) < 1/m$.
Note that
\begin{align*}
{\rm ls} (A-T,3) \leq \alpha d ^{5/4} \left(q_8 - \left(q_4-d\right)^2\right)^{-1/4} \overline{\delta}(A-T),
\end{align*}
Using the method from Section 3.1 
of~\cite{Wa18} applied to our setting, we deduce
\begin{align*}
{\rm ls} (A-T,3) + {\rm ls} (T,3) \geq d ^{5/4} \left(q_8 - \left(q_4-d\right)^2\right)^{-1/4}.
\end{align*}
Combining the two equations above, 
\begin{align*}
{\rm ls} (T,3) &\geq d ^{5/4} \left(q_8 - \left(q_4-d\right)^2\right)^{-1/4} (1- \alpha \overline{\delta}(A-T)),\\
{\rm ls} (T,3)^2 &\geq 
\left(\frac{d ^{5/4}}{ \left(q_8 - \left(q_4-d\right)^2\right)^{1/4}}\right)^2 (1- \alpha)^2 .
\end{align*}
From equation (\ref{k6eq}), we have 
\begin{align*}
{\rm ls} (T,3)^2 \leq {\rm ls} (T,6)\cdot {\rm ls} (T,0) \leq q_6 \cdot \overline{\delta}(T),
\end{align*}
and so
\begin{align}
\left(\frac{d ^{5/4}}{ \left(q_8 - \left(q_4-d\right)^2\right)^{1/4}}\right)^2 (1- \alpha)^2 \leq q_6 \cdot \overline{\delta}(T). \label{deseq2}
\end{align}

Given $\beta$, choose $\alpha$ such that 
\begin{align*}
\left(\left(q_4-d\right)\beta \right)^{1/4} = \left(\alpha \frac{d ^{5/4}}{ \left(q_8 - \left(q_4-d\right)^2\right)^{1/4}}\right)^{1/3}
\end{align*}

We now specify $r \geq 6$.
We therefore can set $q_4 = 3/4$, $q_6 = 25/16$, and $q_8 = 519/128$.
If we choose $\alpha$ and $\beta$ such that the upper Dirichlet densities of the sets $S$ and $T$ are bounded above by $1/234$, then the equations~(\ref{deseq1}) and~(\ref{deseq2}) imply that ($\beta =  0.4906 \dots$, $d= 0.4934 \dots$) is a boundary case. Therefore, there is an upper Dirichlet density of at least $1/234$ for the set of places $v \in A$ such that $${\rm Re}(a_v (\pi)e ^{-i \phi}) >  ((q_4-d)\beta)^{1/4} -\epsilon = 0.59566 \dots - \epsilon,$$
for any $\epsilon > 0$.

Recall that Lemma~\ref{kslem} states that for $Q \geq 2$ we have 
\begin{align*}
\overline{\delta}\{v \mid |a_v (\pi)| > Q\} \leq \frac{1}{1 + (Q ^2 -1)^2 + (Q ^4 -3 Q ^2 + 1)^2}.
\end{align*}
The right-hand side is smaller than $1/234$ when $Q > 2.341$. This implies that there is a positive upper Dirichlet density of places $v$ where $a_v (\pi) e^{-i \phi}$ lies in the region  $$\{z \in \C \mid {\rm Re}(ze ^{-i \phi})>0.59566 , |z| \leq 2.341 \}.$$ Note that $\cos ^{-1} (0.59566 / 2.341) = 1.31352$ radians
(which is equal to $75.259$ degrees). This means that there is a positive upper Dirichlet density of places $v$ whose associated Hecke eigenvalues whose argument is in the interval $$(-1.31353 - \phi , +1.31353 - \phi).$$

\subsection{Case of central characters of order at most five}\label{loworder}
We now assume that the central character $\omega$ of the cuspidal automorphic representation $\pi$ is of order less than six. 

We are handling these cases separately since our bounds for the asymptotic behaviour of various Dirichlet series from Section \ref{dsec} are less strong, and so would lead to a weaker result if we only relied on the proof for the $r \geq 6$ case in the previous two subsections.

At a finite place $v$ where $\pi$ is unramified, we have the associated multiset of Satake parameters $\{\alpha_v (\pi),\beta_v (\pi)\}$ where their product is equal to some (not necessarily primitive) $r$th root of unity $e^{i\mu }$, and their sum is equal to the Hecke eigenvalue $a_v(\pi)$. We write $\alpha_v (\pi) = \rho  e ^{i \theta}$ and $\beta_v(\pi) = \rho  ^{-1} e^{i (-\theta + \mu)}$, for some positive real number $\rho$ and some angle $\theta$.

Unitarity implies that 
\begin{align}\label{unit}
\{\rho e ^{-i \theta}, \rho ^{-1} e ^{i (\theta - \mu) }\}= \{\rho ^{-1} e ^{-i \theta}, \rho e ^{i (\theta - \mu) }\}.
\end{align}
If $\rho = 1$, then 
\begin{align*}
{\rm Re}(a_v(\pi)) = (1 + \cos \mu) \cos \theta + \sin \mu \sin \theta  \\
{\rm Im} (a_v(\pi)) = (1- \cos \mu)\sin \theta   + \sin \mu \cos \theta 
\end{align*}
and 
\begin{align*}
\frac{{\rm Im} (a_v(\pi))}{{\rm Re} (a_v(\pi))}= \frac{\sin \mu}{1 + \cos \mu}= \tan (\mu /2) ,
\end{align*}
 so ${\rm arg}(a_v(\pi)) = \mu /2 + n\pi$, for some integer $n$.\\
If $\rho \neq 1$, then equation (\ref{unit}) implies $e^{-i \theta}= e^{i (\theta - \mu)}$, so $\theta = \mu /2 + n \pi$ for some integer $n$. This again means 
\begin{align}\label{angleeq}
{\rm arg}(a_v(\pi)) = \mu /2 + n \pi.
\end{align}

We also want to apply the method of Subsection~\ref{mpf}. For each $r$ (and corresponding $q_4,q_6$ and $q_8$), we obtain a statement, for any angle $\phi \in [0,2 \pi)$, of the form 
\begin{align}\label{deqn}
\overline{\delta}(\{v \mid {\rm Re}(a_v(\pi) e ^{-i \phi})> T(r)\})> 0.
\end{align}
For $r = 5$, we set $q_4 = 3/4 $, $q_6 = 25/16$, and $q_8 = 583/128$, and obtain $T(5)=0.679$.
For $r = 4$, we set $q_4 = 3/4$, $q_6 = 25/16$, and $q_8 = 504/128$, and get $T(4) = 0.684$.
For $r = 3$, set $q_4 = 3/4$, $q_6 = 55/32$, and $q_8 = 602/128$, obtaining $T(3) = 0.678$.\\
In the case of $r = 2$, we set $q_4 = (3 + \cos 4 \phi)/4 $, $q_6 = 5/2$, and $q_8 = 7$, and obtain, for $\cos 4 \phi \geq -0.785$,
\begin{align*}
\overline{\delta}(\{v \mid {\rm Re}(a_v(\pi) e ^{-i \phi})> 0.5956\})> 0.
\end{align*}
If $\cos 4 \phi < -0.785$, then we  conclude that 
\begin{align*}
\overline{\delta}(\{v \mid {\rm Re}(a_v(\pi) e ^{-i \phi})> 0.5723 \})> 0
\end{align*}
and use of basic geometry in this setting then implies that $|a_v(\pi)|> 0.702$.

Applying the results from the above equations \ref{angleeq} and \ref{deqn} for suitable values of $r$ and $\phi$, we find that any sector of angle greater than $144^\circ$ (i.e., 2.51 radians) must contain a positive upper Dirichlet density of Hecke eigenvalues of size greater than $0.5956$, which proves Theorem~\ref{t1} for $r \leq 5$.

\subsection{Acknowledgements} The author would like to thank Dinakar Ramakrishnan for suggesting this problem.

\end{document}